\documentclass[12pt]{article}
\usepackage{pifont}
\usepackage{amsmath}
\usepackage{amssymb,float}
\usepackage{amsthm,graphicx,pifont}
\usepackage{graphicx,psfrag,float,subfigure}
\usepackage{caption,enumerate}
\usepackage[scale=0.8,a4paper]{geometry}
\usepackage[colorlinks]{hyperref}
\usepackage{amsthm,amsmath,amssymb}
\usepackage{mathrsfs}
\newtheorem{theorem}{Theorem}[section]

\newtheorem{lemma}[theorem]{Lemma}

\newtheorem{prop}[theorem]{Proposition}

\title{Claw-free minimal matching covered graphs}
\author{{\small\bf Yipei Zhang$^{1}$, Xiumei Wang$^{2}$\thanks{Corresponding author. Email address: wangxiumei@zzu.edu.cn}, Jinjiang Yuan$^{2}$, C.T. Ng$^{3}$, T.C.E. Cheng$^{3}$}\\
{\small $^{1}$School of Mathematics and Statistics, North China University of Water Resources and Electric Power,}\\
{\small Zhengzhou, Henan 450046, People's Republic of China}\\
{\small $^{2}$School of Mathematics and Statistics, Zhengzhou University,}\\
{\small Zhengzhou, Henan 450001, People's Republic of China}\\
{\small $^{3}$Logistics Research Centre, Department of Logistics and Maritime Studies,}\\
{\small The Hong Kong Polytechnic University, Hong Kong SAR, People's Republic of China}}
\date{}

\makeatother
\begin{document}
\maketitle

\begin{abstract}
A matching covered graph $G$ is minimal if for each edge $e$ of $G$, $G-e$ is not matching covered. An edge $e$ of a matching covered graph $G$ is removable if $G-e$ is also matching covered. Thus a matching covered graph is minimal if and only if it is free of removable edges.
For bipartite graphs, Lov\'{a}sz and Plummer gave a characterization of bipartite minimal matching covered graphs.
For bricks, Lov\'{a}sz showed that the only bricks that are minimal matching covered are $K_4$ and $\overline{C_6}$.
In this paper, we present a complete characterization of minimal matching covered graphs that are claw-free.
Moreover, for cubic claw-free matching covered graphs that are not minimal matching covered, we obtain the number of their removable edges (with respect to their bricks), and then prove that they have at least 12 removable edges (the bound is sharp).
\end{abstract}


\section{Introduction}
All the graphs considered in this paper may have multiple edges (except those considered in Theorem \ref{2-connected-cubic-claw-free} and Section 4), but no loops. We follow \cite{BM08} for undefined notation and terminology.
Let $G$ be a graph with the vertex set $V(G)$ and edge set $E(G)$.
For a subset $X$ of $V(G)$, we denote by $G[X]$ the subgraph of $G$ induced by $X$.
We denote by $K_n$ the complete graph with $n$ vertices.
An edge $e$ of $G$ is \emph{admissible} if $G$ has a perfect matching that contains $e$; otherwise, it is \emph{inadmissible}.
A connected nontrivial graph $G$ is \emph{matching covered} if each edge of $G$ is admissible. Clearly, every matching covered graph different from $K_2$ is 2-connected.
We say that a matching covered graph $G$ is {\it minimal} if for each edge $e$ of $G$, $G-e$ is not matching covered.
Let $G$ and $H$ be two graphs. We say that $G$ is $H$-{\it free} if $G$ does not contain $H$ as an induced subgraph. In particular, a graph $G$ is called {\it claw-free} if it is $K_{1,3}$-free.

For a nonempty proper subset $X$ of $V(G)$, we denote by $\partial(X)$ the set of all the edges of $G$ with one end in $X$ and the other end in $\overline{X}$, where $\overline{X}=V(G)\setminus X$. We refer to $\partial(X)$ as a {\it cut} of $G$.
A cut $\partial(X)$ is {\it trivial} if one of $X$ and $\overline{X}$ has exactly one vertex, and is \emph{nontrivial} otherwise.
A cut $\partial(X)$ is a {\it k-cut} if $|\partial(X)|=k$.
For a cut $C=\partial(X)$ of a graph $G$, we denote by
$G\{X,\overline{x}\}$, or simply $G\{X\}$, the graph obtained from $G$ by shrinking $\overline{X}$ to a single vertex $\overline{x}$.
The graph $G\{\overline{X},x\}$ or simply $G\{\overline{X}\}$, is defined analogously.
The two graphs $G\{X\}$ and $G\{\overline{X}\}$ are called the two $C$-{\it contractions} of $G$.

Let $G$ be a matching covered graph. A cut $C$ is a \emph{tight cut} if $|M\cap C|=1$ for each perfect matching $M$ of $G$. A matching covered graph that is free of nontrivial tight cuts is a \emph{brace} if it is bipartite, and is a \emph{brick} otherwise.
If $C$ is a nontrivial tight cut of $G$, then the two $C$-contractions of $G$ are also matching covered.
Continuing  in this way, we can obtain a list of graphs without nontrivial tight cuts, which are bricks and braces.
This procedure is known as a {\it tight cut decomposition} of $G$.  In general, a matching covered graph may admit several tight cut decompositions.  Lov\'asz \cite{Lovasz87} showed that any two tight cut decompositions of a matching covered graph yield the same list of bricks and braces (up to multiple edges).

A graph $G$ with at least four vertices is {\it bicritical} if, for any two distinct vertices $u$ and $v$ of $G$, the subgraph $G-\{u,v\}$ has a perfect matching. For each bicritical graph $G$, one may easily verify that $\delta(G)\geq3$ and $G$ is matching covered. Edmonds et al. \cite{ELP82} showed that a graph is a brick if and only if it is 3-connected and bicritical.

\begin{figure}[h]
 \centering
 \includegraphics[width=0.55\textwidth]{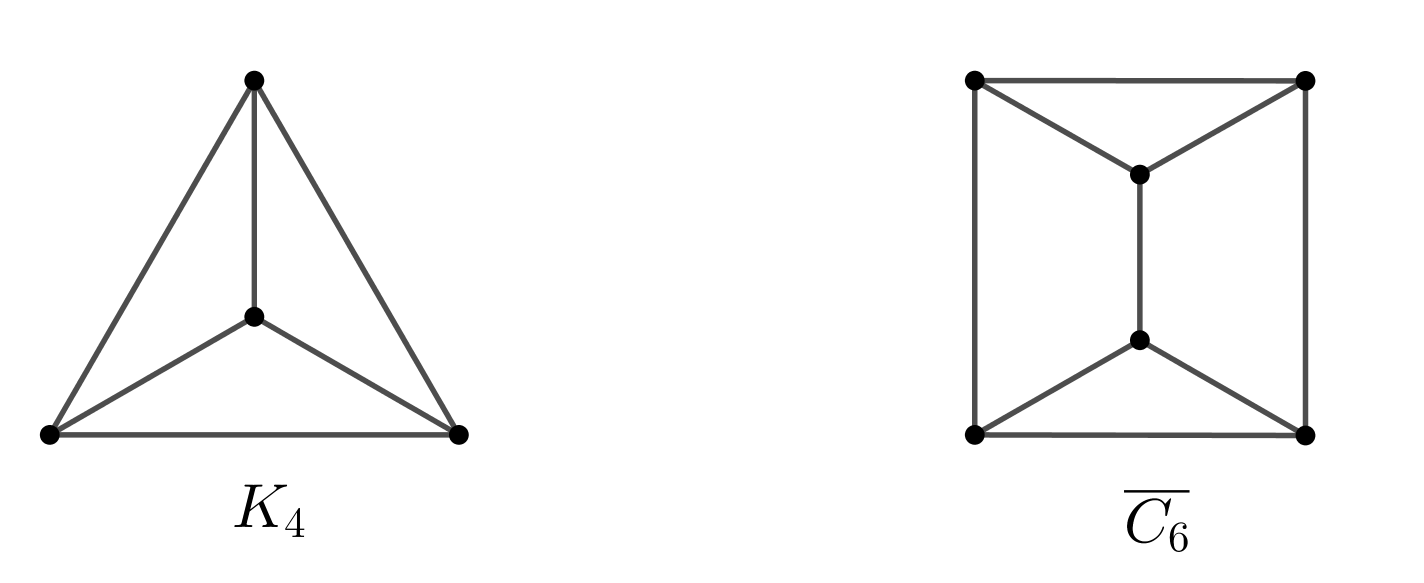}\\
 \caption{The two bricks.}\label{fig1}
\end{figure}

An edge $e$ of a matching covered graph $G$ is \emph{removable} if $G-e$ is also matching covered, and is \emph{nonremovable} otherwise.
The notion of removable edge arises in connection with ear decompositions of
matching covered graphs introduced by Lov\'{a}sz and Plummer.
Clearly,  each multiple edge of a matching covered graph is in fact a removable edge.
Note that a matching covered graph is minimal if and only if it is free of removable edges.
For a bipartite  graph $G$,  Lov\'{a}sz and Plummer \cite{LP1977} proved that $G$ is minimal matching covered if and only if no nice cycle of $G$ has a chord.
For bricks, Lov\'asz \cite{Lovasz87} showed that
the only bricks that are minimal matching covered are $K_4$ and $\overline{C_6}$, which are depicted in Figure \ref{fig1}.
Zhang et al. \cite{ZYP22} presented a characterization of minimal matching covered graphs which are bicritical. Using this characterization, we can obtain all the minimal matching covered graphs which are claw-free and bicritical (Theorem \ref{claw-free-bicriticalR(G)emptyset}).
In this paper, we give a complete characterization of minimal matching covered graphs which are claw-free.

\begin{theorem}[\cite{ZYP22}] \label{claw-free-bicriticalR(G)emptyset}
A graph $G$ is a minimal matching covered graph which is claw-free and bicritical if and only if $G\in \mathcal{G}$. 
\end{theorem}

To describe the graphs in $\mathcal{G}$, we first give some definitions. An edge of a graph  that does not lie in a triangle is called a {\it ridge}. For convenience, we say that each edge of $K_4$ is a ridge.
Let $K_4^{-}$ and $\overline{C_6}^{*}$ be the graphs obtained from $K_4$ and $\overline{C_6}$ by deleting a ridge, respectively.
Note that each of $K_4^{-}$ and $\overline{C_6}^{*}$ has exactly two vertices of degree 2.
To replace an edge $u_1u_2$ of a graph $G$ by $K_4^{-}$ or $\overline{C_6}^{*}$ with two vertices of degree 2, say $v_1$ and $v_2$, is to delete the edge $u_1u_2$ from $G$ and then identify $u_i$ and $v_i$ into a new vertex, $i=1,2$.

Let $\mathcal{G}$ be the set of graphs obtained from $K_4$ or $\overline{C_6}$ by  recursively  replacing some ridges by  $K_4^{-}$ or $\overline{C_6}^{*}$.
Note that $K_4$, $\overline{C_6}\in \mathcal{G}$,
and the two graphs shown in Figure \ref{fig2} belong to $\mathcal{G}$.
Moreover, if $G\in\mathcal{G}$ and $x\in V(G)$, then (i) $G$ is either $K_4$-free or isomorphic to $K_4$,  (ii) $d(x)=3$ or 4, and (iii) if $d(x)=4$, then $G[N(x)]$ is the union  of two disjoint $K_2$s.
Let $\mathcal{F}$ be the set of graphs that lie in $\mathcal{G}$ or are compounds (defined in Section 2) of two graphs in  $\mathcal{F}$.

\begin{figure}[h]
 \centering
 \includegraphics[width=0.6\textwidth]{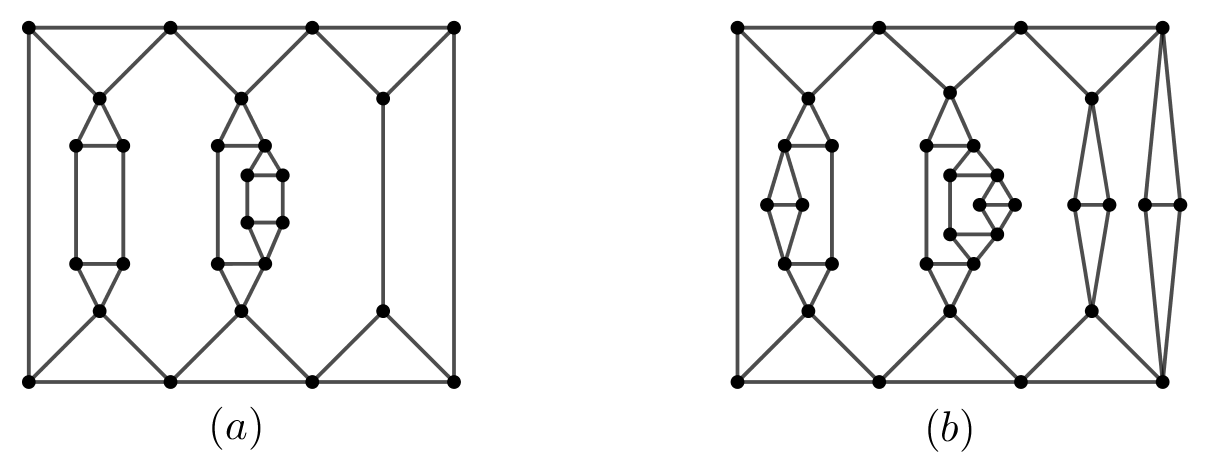}
 \caption{The two graphs in  $\mathcal{G}.$}\label{fig2}
\end{figure}

A path is odd or even according to the parity of its length.
Let $H$ be a graph. {\it To bisubdivide an edge $e$} of $H$ is to replace $e$ by an odd path with length at least three. 
The resulting graph is called a {\it  bisubdivision} of $H$ at the edge $e$.
A vertex of $H$ is {\it bisimplicial} if its neighbours induce  two disjoint cliques such that each clique has at least two vertices.
{\it To expand  a bisimplicial vertex $u$} of $H$ is to (i) split $u$ into two vertices $u'$ and $u''$, which are called {\it split vertices},  (ii) join $u'$ to each vertex of one clique in $H[N(u)]$ and $u''$ to each vertex of the other clique, and (iii) add an even path with length at least two that connects  $u'$ and $u''$.
A graph is called a  {\it $u$-splitting} of $H$ at a bisimplicial vertex $u$ if it is obtained from $H$ only by operations (i) and (ii).
A graph is called an {\it expansion of $H$} if it is obtained from $H$ by recursively expanding bisimplicial vertices or bisubdividing  ridges.
For convenience, we say that $H$ is  an expansion of itself.
Often  we may use subscripts  to denote a vertex, e.g., $u_1$.
In this case, the two  split  vertices obtained from $u_1$ will be decorated similarly, that is, $u_1'$ and $u_1''$.
The following theorem gives a complete characterization of minimal matching covered graphs that are claw-free. 

\begin{theorem} \label{claw-free-mc}
A graph $G$ different from $K_2$ and even cycle is a claw-free minimal matching covered graph if and only if $G$ is an expansion of a graph in $\mathcal{F}$.
\end{theorem}

Note that for a cubic claw-free matching covered graph $G$, if $G$ is not minimal matching covered, then $G$ has removable edges. In fact, we can prove the following stronger result.
Since each multiple edge of a matching covered graph is a removable edge, we only consider simple graphs in Theorem \ref{2-connected-cubic-claw-free}. To describe Theorem \ref{2-connected-cubic-claw-free}, we give some notation.
For a matching covered graph $G$, we denote by $RE(G)$  the set of all the removable edges of $G$, and by $b^{*}(G)$  the number of the bricks of $G$ each of which is different from $K_4$ and $\overline{C_6}$.
If $b^{*}(G)\geq1$, the number of the vertices of the $b^{*}(G)$ bricks is denoted by $n_1, n_2,\ldots, n_{b^{*}(G)}$, respectively.

\begin{theorem}\label{2-connected-cubic-claw-free}
Let $G$ be a cubic claw-free matching covered graph. If $G$ is not minimal matching covered, then the following statements hold.
\vspace{-8pt}
\begin{enumerate}[(\romannumeral1)]
\setlength{\itemsep}{-1ex}
\item  $b^{*}(G)\geq1$ and $|RE(G)|=\sum_{i=1}^{b^{*}(G)}n_i$.
\item  $G$ has at least 12 removable edges, and the bound is sharp.
\end{enumerate}
\end{theorem}

We organize the rest of the paper as follows. In Section 2, we present some basic results and operations. In Sections 3 and 4, we prove Theorems \ref{claw-free-mc} and \ref{2-connected-cubic-claw-free}, respectively.

\section{Preliminaries}

\subsection{Basic results}

\begin{lemma}[\cite{Tutte47}]\label{Tutte}
A graph $G$ has a perfect matching if and only if $o(G-S)\leq|S|$ for every $S\subseteq V(G)$, where $o(G-S)$ denotes the number of odd components of $G-S$.
\end{lemma}

Let $G$ be a graph with a perfect matching. A subset $B$ of $V(G)$ is a {\it barrier} of $G$ if $o(G-B)=|B|$.  Clearly, the empty set is a barrier. In this paper, by a barrier we mean a nonempty barrier.
All the singletons are barriers of $G$, which are referred to as {\it trivial} barriers. A barrier that has exactly $k$ vertices, where $k\geq2$, is called a {\it $k$-barrier}.
A graph $H$ is {\it factor-critical} if, for any vertex $u$ of $H$, the subgraph $H-u$ has a perfect matching. By Lemma \ref{Tutte}, we can easily prove the following four propositions.

\begin{prop}\label{2-barrier}
Let $G$ be a connected graph with a perfect matching and $\{u,v\}$ a 2-vertex cut of $G$. If $G-\{u,v\}$ has an odd component, then $\{u,v\}$ is a 2-barrier of $G$.
\end{prop}

\begin{prop}\label{prop2}
Let $G$ be a graph with a perfect matching. Then
\vspace{-8pt}
\begin{enumerate}[(\romannumeral1)]
\setlength{\itemsep}{-1ex}
\item  an edge $e$ of $G$ is admissible if and only if there is no barrier that contains both ends of $e$, and
\item  for each  maximal barrier $B$ of $G$, all the components of $G-B$ are factor-critical.
\end{enumerate}
\end{prop}

\begin{prop}\label{mc-B-stableset}
Let $G$ be a graph with a perfect matching. Then $G$ is matching covered if and only if, for every barrier $B$ of $G$, $G-B$ has no even components and $B$ is a stable set.
\end{prop}

\begin{prop}\label{bicritical-no-barrier}
 A matching covered graph different from $K_2$ is bicritical if and only if it is free of nontrivial barriers.
\end{prop}

\begin{lemma}\label{clawfreebicritical-no2barrier}
 A claw-free matching covered graph different from $K_2$ is bicritical if and only if it is free of 2-barriers.
\end{lemma}

\begin{proof}
Let $G$ be a claw-free matching covered graph different from $K_2$. If $G$ is bicritical, by Proposition \ref{bicritical-no-barrier}, the result holds.

Conversely, suppose that $G$ is free of 2-barriers. Let $B$ be a barrier of $G$. Then $|B|\neq2$. By Proposition \ref{mc-B-stableset}, $G-B$ has no even components and $B$ is a stable set. Now suppose that $|B|\geq3$. Let $H$ be the underlying simple graph of the graph obtained from $G$ by contracting each component of $G-B$ to a single vertex.
Note that $G$ is 2-connected. Then each component of $G-B$ connects at least two vertices in $B$, so $|E(H)|\geq2|B|$. On the other hand, since $G$ is claw-free, each vertex of $B$ connects at most two components of $G-B$, so $|E(H)|\leq2|B|$. Therefore, $|E(H)|=2|B|$. This implies that each component $L$ of $G-B$ connects exactly two vertices in $B$, say $\{u_L,v_L\}$. Then $\{u_L,v_L\}$ is a 2-vertex cut of $G$, and $G-\{u_L,v_L\}$ has an odd component $L$. By Proposition \ref{2-barrier},  $\{u_L,v_L\}$ is a 2-barrier of $G$, a contradiction. Thus $|B|=1$. By Proposition \ref{bicritical-no-barrier}, $G$ is bicritical.
\end{proof}

\begin{lemma}[\cite{ZYP22}]\label{removable-tight}
Let $C$ be a tight cut of a matching covered graph $G$ and $e$ an edge of $G$. Then $e$ is
removable in $G$ if and only if $e$ is removable in each $C$-contraction of $G$ that contains it.
\end{lemma}

Let $G$ be a matching covered graph, and let $e$ and $f$ be  two distinct  edges of $G$. We say that $e$ {\it depends on} $f$, if every perfect matching of $G$ that contains $e$ also contains $f$. Equivalently, $e$ depends on $f$ if $e$ is inadmissible in $G-f$.
We call $e$ and $f$ are {\it mutually dependent} if $e$ depends on $f$ and vice versa. 
Observe that mutual dependence is an equivalence relation on $E(G)$, so the edge set $E(G)$ can be partitioned into equivalence classes. In general, an equivalence class of $G$ can be arbitrarily large, see \cite{Lu20}. However, Carvalho et al. \cite{CLM99} showed that the  cardinality of each equivalence class of a brick is at most two. The following lemma can be easily derived.

\begin{lemma}\label{depend-on}
For every matching covered graph $G$ different from $K_2$, an edge $e$ is removable in $G$ if and only if no other edge depends on $e$.
\end{lemma}

\begin{lemma}\label{K5freeK4}
Let $G$ be a matching covered graph that contains a complete subgraph $H$.
\vspace{-8pt}
\begin{enumerate}[(\romannumeral1)]
\setlength{\itemsep}{-1ex}
\item  If $|V(H)|\geq5$, then at least one of any two adjacent edges of $H$ is removable in $G$. In particular, if $G$ has no removable edges, then $G$ is $K_5$-free.
\item  Suppose that there is a vertex $u$ of $H$ that has exactly one neighbour $w$ in $V(G)\backslash V(H)$.
If either  $|V(H)|=4$ and each edge of $H$ is nonremovable in $G$ or $|V(H)|=3$, then $uw$ is nonremovable in $G$.
\end{enumerate}
\end{lemma}

\begin{proof}
Let $V(H)=\{v_1,v_2,\ldots,v_k\}$, where $k\geq1$. Firstly, we show the following claim.
\vspace{1.5mm}

\textbf{Claim.} If $e$ is nonremovable in $G$, then there is a maximal barrier $B$ of $G-e$ such that $e$ has its ends in distinct components of $G-e-B$, and each edge whose two ends lie in $B$ depends on $e$.

Suppose that $e$ is nonremovable in $G$. Then $G-e$ is not matching covered, so it has an inadmissible edge $f$. By Proposition \ref{prop2}(i), $G-e$ has a barrier that contains the two ends of $f$. Let $B$ be such a maximal barrier. By Proposition \ref{prop2}(ii), each component of $G-e-B$ is factor-critical (and odd). Since $f$ is admissible in $G$, $e$ has its ends in distinct components of $G-e-B$. Observe that for each edge $h$ whose two ends lie in $B$, each perfect matching of $G$ that contains $h$ also contains $e$. So $h$ depends on $e$. The claim follows.
\vspace{1.5mm}

(i) Suppose that $k\geq5$. By symmetry, it suffices to prove that at least one of $v_1v_i$ and $v_1v_j$ is removable in $G$, where $2\leq i,j\leq k$ and $i\neq j$. Since $k\geq5$, there are at least two vertices in $V(H)\backslash \{v_1,v_i,v_j\}$, say $s$ and $t$. Assume, to the contrary, that both $v_1v_i$ and $v_1v_j$ are nonremovable in $G$. By the above Claim, there is a maximal barrier $B_l$ of $G-v_1v_l$ such that $v_1v_l$ has its ends in distinct components of $G-v_1v_l-B_l$, and each edge whose two ends lie in $B_l$ depends on $v_1v_l$, $l=i,j$. Recall that $H$ is complete. Then $\{s,t\}\subseteq B_i\cap B_j$. It follows that $st$ depends on $v_1v_i$ and $v_1v_j$, that is,
each perfect matching of $G$ that contains $st$ also contains $v_1v_i$ and $v_1v_j$. However, it is impossible. Thus either $v_1v_i$ or $v_1v_j$ is removable in $G$. (i) holds.

(ii) Suppose, without loss of generality, that $u=v_1$. If $k=3$, then each perfect matching of $G$ that contains $v_2v_3$ also contains $uw$, that is, $v_2v_3$ depends on $uw$. By Lemma \ref{depend-on}, $uw$ is nonremovable in $G$.

Now we may assume that $k=4$ and each edge of $H$ is nonremovable in $G$. By the above Claim, there is a maximal barrier $B'$ of $G-v_1v_2$ such that $v_1v_2$ has its ends in distinct components of $G-v_1v_2-B'$, and each edge whose two ends lie in $B'$ depends on $v_1v_2$. Since $H$ is complete, we have $v_3,v_4\in B'$, so $v_3v_4$ depends on $v_1v_2$. Likewise, $v_1v_2$ depends on $v_3v_4$. Therefore, $v_1v_2$ and $v_3v_4$ are mutually dependent. By symmetry, $v_1v_3$ and $v_2v_4$ are mutually dependent, and $v_1v_4$ and $v_2v_3$ are mutually dependent. Since $G$ is matching covered, there is a perfect matching $M$ containing $uw$. Clearly, $M\cap \{v_1v_2,v_1v_3,v_1v_4\}=\emptyset$. It follows that $M\cap E(H)=\emptyset$. Let $v_2v_2'\in M$. Then $v_2'\notin V(H)$. Since $u$ has exactly one neighbour $w$ in $V(G)\backslash V(H)$, each perfect matching of $G$ that contains $v_2v_2'$ also contains $uw$. So $v_2v_2'$ depends on $uw$. By Lemma \ref{depend-on}, $uw$ is nonremovable in $G$.
\end{proof}

\subsection{Operations}

Let $G_1$ and $G_2$ be two disjoint graphs with a specified vertex $u_1$ and $u_2$, respectively, such that $d_{G_1}(u_1)=d_{G_2}(u_2)$. Suppose that $\pi$ is a bijection between  $\partial _{G_1}(u_1)$ and $\partial _{G_2}(u_2)$.  Let $G$ be  the graph obtained from the union of $G_1-u_1$ and $G_2-u_2$ by joining, for edge $e$ in $\partial _{G_1}(u_1)$, the end of $e$ in $G_1-u_1$ to the end of $\pi(e)$ in $G_2-u_2$. We refer to $G$ as a graph obtained by \emph{splicing} $G_1$ and $G_2$.
The following lemma is easily proved by the definition of
matching covered graph.

\begin{lemma}\label{splicing}
Any graph obtained by splicing two matching covered graphs is also matching covered.
\end{lemma}

Let $G$ be a graph and $\partial(X)$ an edge cut of $G$. Then $G$ is obtained by splicing $G\{X\}$ and $G\{\overline{X}\}$. By Lemma \ref{splicing}, if both $G\{X\}$ and $G\{\overline{X}\}$ are matching covered, then $G$ is matching covered. Thus we can construct bigger matching covered graphs from smaller matching covered graphs by the operation of splicing.

Recall that a vertex of a graph is {\it bisimplicial} if its neighbours induce two disjoint cliques such that each clique has at least two vertices. Analogously, we say that an edge $ab$ of a graph $G$ is {\it bisimplicial} if $N(a)\backslash \{b\}$ and $N(b)\backslash \{a\}$ are  two disjoint cliques and each of them has at least two vertices.
For convenience,  we say that each edge of $K_4$ is bisimplicial.
To construct bigger claw-free matching covered graphs from smaller ones, we now define the following four operations with respect to bisimplicial vertices and bisimplicial edges (see Figure \ref{fig3}).  Let $G_1$ and $G_2$ be two disjoint connected graphs, and each of them has even number of vertices. Let $u_i$ be a bisimplicial vertex of $G_i$,   let $e_i=a_ib_i$ be a bisimplicial edge of $G_i$, and let $G_i'$ be a $u_i$-splitting of $G_i$ at $u_i$.
Then  $u_i'$ and $u_i''$ are the two split vertices of $G_i'$ obtained from $u_i$, $i=1,2$. 

\begin{figure}[h]
 \centering
 \includegraphics[width=\textwidth]{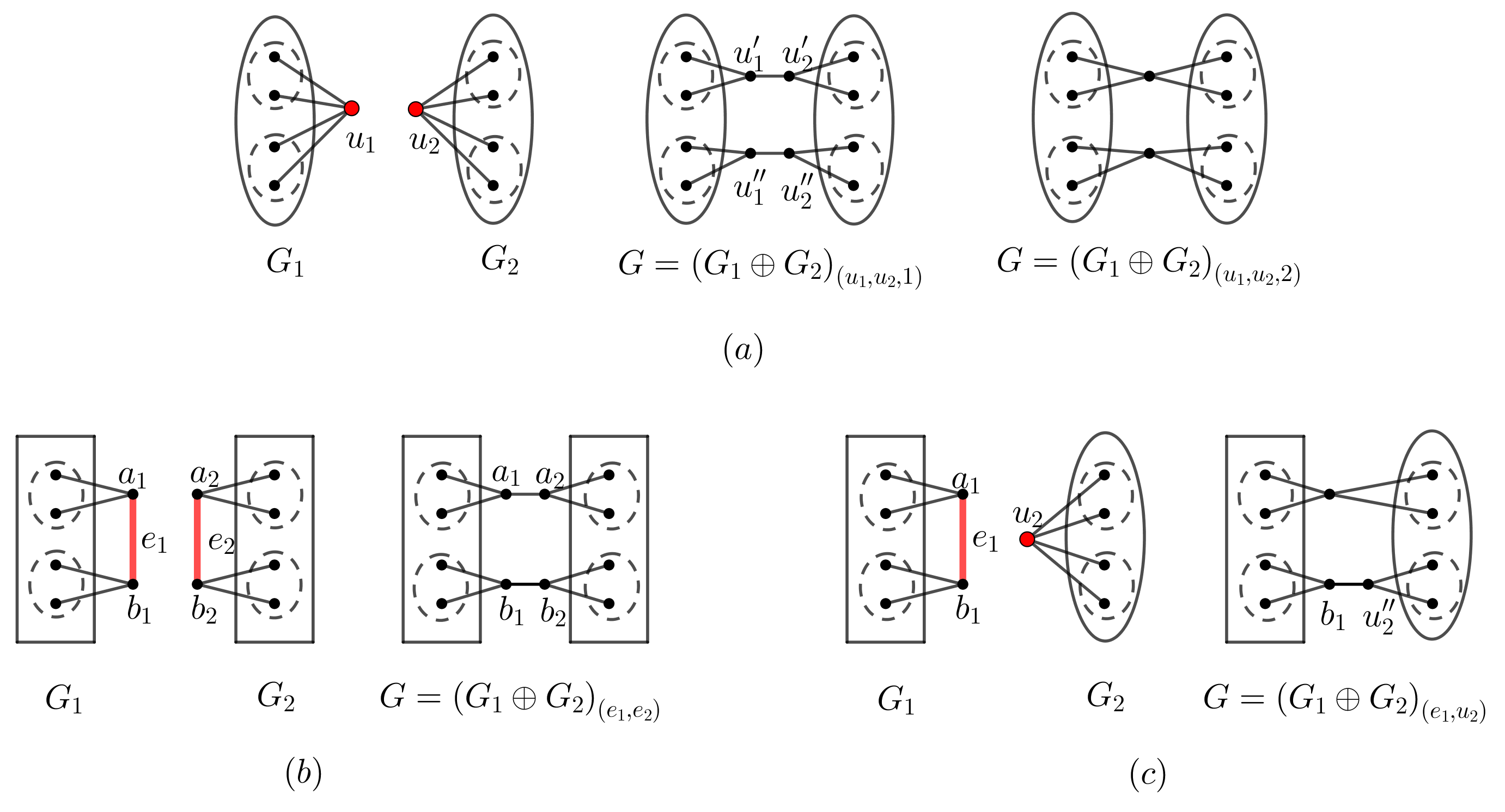}\\
 \caption{The four operations, where every dashed circle denotes a clique.}\label{fig3}
\end{figure}

If $G$ is the graph obtained from the union of $G_1'$ and $G_2'$ by adding two new edges $u_1'u_2'$ and $u_1''u_2''$, then we say that  $G$ is obtained from $G_1$ and $G_2$ by {\it V-joining} (at $u_1$ and $u_2$), and write   $G=(G_1\oplus G_2)_{(u_1,u_2,1)}$, see Figure \ref{fig3}(a).

If $G$ is the graph obtained from the union of $G_1'$ and $G_2'$ by identifying $u_1'$ and $u_2'$, and $u_1''$ and $u_2''$ into new vertices, respectively, then we say that  $G$ is obtained from $G_1$ and $G_2$ by {\it V-attaching} (at $u_1$ and $u_2$), and write  $G=(G_1\oplus G_2)_{(u_1,u_2,2)}$, see Figure \ref{fig3}(a).

If $G$ is the graph obtained from the union of  $G_1$ and $G_2$ by
(i) deleting the two edges $e_1$ and $e_2$, and
(ii) adding two new edges $a_1a_2$ and $b_1b_2$, then we say that $G$ is obtained from $G_1$ and $G_2$ by {\it E-joining} (at $e_1$ and $e_2$), and write $G=(G_1\oplus G_2)_{(e_1,e_2)}$, see Figure \ref{fig3}(b).

If  $G$ is the graph obtained from the union of $G_1$ and $G_2'$ by
(i) deleting the edge $a_1b_1$,
(ii) identifying $a_1$ and $u_2'$ into a single vertex, and
(iii) adding a new edge $b_1u_2''$, then we say that the graph $G$ is obtained from $G_1$ and $G_2$ by {\it EV-attaching} (at $e_1$ and $u_2$), and write  $G=(G_1\oplus G_2)_{(e_1,u_2)}$, see Figure \ref{fig3}(c).


Let $G$ be a graph  obtained from $G_1$ and $G_2$ by one of the above four operations. Then we say $G$ is a {\it compound} of   $G_1$ and $G_2$.  Note that each  edge of $G_1$ and $G_2$ other than  two possible edges $e_1$ and $e_2$ corresponds to an edge of $G$. When no confusion occurs, they are viewed  as the same edge.

Recall that each graph $G$ in $\mathcal{G}$ has the following properties:
(i) $G$ is either $K_4$-free or isomorphic to $K_4$,
(ii) each vertex of $G$ has degree 3 or 4, and
(iii) for a vertex $x\in V(G)$, if $d(x)=4$, then $G[N(x)]$ is the union  of two disjoint $K_2$s.
For each graph that is an expansion of a graph in $\mathcal{F}$, by the properties of graphs in  $\mathcal{G}$  and by the definitions of  expansion and the set $\mathcal{F}$, we can easily derive the following lemma.

\begin{lemma}\label{4-degree-2K2}
Let $G$ be an expansion of a graph in $\mathcal{F}$ and $x\in V(G)$. Then
\vspace{-8pt}
\begin{enumerate}[(\romannumeral1)]
\setlength{\itemsep}{-1ex}
\item  $G$ is either  $K_4$-free or isomorphic to $K_4$,
\item  $2\leq d(x)\leq4$, and
\item  if $d(x)=4$, then $G[N(x)]$ is the union of two disjoint $K_2$s.
\end{enumerate}
\end{lemma}

Let $G$ be a matching covered graph and $B$ a barrier of $G$.  If $L$ is a component of $G-B$, then $\partial (V(L))$ is a tight cut of $G$,  such a tight cut is referred to as a {\it barrier cut}. Recall that  $RE(G)$ denotes the set of all the removable edges of $G$. The following lemma describes an important property with respect to the compound of two graphs.

\begin{lemma}\label{four-operation-property}
Let $G$ be a compound of two graphs $G_1$ and $G_2$. Then the following statements hold.
\vspace{-8pt}
\begin{enumerate}[(\romannumeral1)]
\setlength{\itemsep}{-1ex}
\item  $G$ is claw-free and matching covered if and only if both $G_1$ and $G_2$ are claw-free and matching covered.
\item  Suppose that $G$ is claw-free and matching covered.
\vspace{-8pt}
\begin{enumerate}[(a)]
\setlength{\itemsep}{-1ex}
\item  If $G=(G_1\oplus G_2)_{(u_1,u_2,k)}$, $k=1,2$, then $RE(G)=RE(G_1)\cup RE(G_2)$, and $RE(G_i)=RE(G)\cap E(G_i)$, $i=1,2$.
\item  If $G=(G_1\oplus G_2)_{(e_1,e_2)}$, then $RE(G)=(RE(G_1)\cup RE(G_2))\backslash \{e_1,e_2\}$,  and\\ $RE(G_i)\backslash \{e_i\}=RE(G)\cap E(G_i)$, $i=1,2$.
\item  If $G=(G_1\oplus G_2)_{(e_1,u_2)}$, then $RE(G)=(RE(G_1)\cup RE(G_2))\backslash \{e_1\}$, $RE(G_1)\backslash \{e_1\}=RE(G)\cap E(G_1)$, and $RE(G_2)=RE(G)\cap E(G_2)$.
\end{enumerate}
\end{enumerate}
\end{lemma}

\begin{proof} Recall that each of $G_1$ and $G_2$ has even number of vertices. We only prove the case when $G=(G_1\oplus G_2)_{(u_1,u_2,1)}$.  The other three cases can be proved by similar arguments.

Suppose that $G=(G_1\oplus G_2)_{(u_1,u_2,1)}$. Then $G$ is not isomorphic to $K_2$ and $\{u_1',u_1''\}$ is a 2-vertex cut of $G$. Let $X_i=V(G_i)\backslash\{u_i\}$, $i=1,2$. Then $G_i$ is isomorphic to $G\{X_i\}$, $i=1,2$. Let $G_3=G\{\overline{X_1}\}$ and $G_4=G_3\{\overline{X_2}\}$.
Then $G_4$ is the graph obtained from $G$ by shrinking $X_1$ and $X_2$ to a single vertex, respectively.
Therefore, the underlying simple graph of $G_4$ is $C_6$, so $G_4$ is matching covered. Moreover, $G_3$ can be obtained by splicing $G_4$ and $G_2$, and $G$ can be obtained by splicing $G_1$ and $G_3$.

(i) Assume that both $G_1$ and $G_2$ are claw-free and matching covered.
By Lemma \ref{splicing} twice, $G$ is matching covered. Recall that $u_i$ is a bisimplicial vertex of $G_i$, $i=1,2$, then $G$ has no claw with centre $u_1',u_1'',u_2'$ and $u_2''$, respectively. So  $G$ is claw-free.

Conversely, assume that $G$ is claw-free and matching covered. Then $G$ is 2-connected. We assert that $G_1-u_1$ is connected. Otherwise, since $N_G(u_1')\cap X_1$ and $N_G(u_1'')\cap X_1$ are  two disjoint cliques, $G$ is disconnected or it has cut vertices $u_1'$ and $u_1''$, contradicting the fact that $G$ is 2-connected. Thus $G_1-u_1$ is an odd component of $G-\{u_1',u_1''\}$. By Proposition \ref{2-barrier}, $\{u_1',u_1''\}$ is a 2-barrier of $G$. So $\partial_G(X_1)$ is a barrier cut of $G$, and then is a tight cut of $G$. Recall that $G_1$ is isomorphic to $G\{X_1\}$. So $G_1$ is matching covered. Since $G$ is claw-free and $G_1$ has no claw with centre $u_1$, $G_1$ is also claw-free. By symmetry, $G_2$ is a claw-free matching covered graph. (i) holds.

(ii) Suppose that $G$ is claw-free and matching covered.
Since $\partial_G(X_1)$ is a  tight cut of $G$, if $f\in RE(G)\cap E(G_1)$, Lemma \ref{removable-tight} implies that $f\in RE(G_1)$. So $RE(G)\cap E(G_1)\subseteq RE(G_1)$.
Assume that $f\in RE(G_1)$. Note that each edge of $\partial_G(X_1)$ is a multiple edge of $G\{\overline{X_1}\}$, so is removable in  $G\{\overline{X_1}\}$.
Again by Lemma \ref{removable-tight},  whether $f\in \partial_G(X_1)$ or not, we have $f\in RE(G)\cap E(G_1)$. Thus $RE(G_1)\subseteq RE(G)\cap E(G_1)$.  Consequently, $RE(G_1)=RE(G)\cap E(G_1)$.
By symmetry, we have $RE(G_2)=RE(G)\cap E(G_2)$.

We next show that  $RE(G)=RE(G_1)\cup RE(G_2)$.
Let $w\in N_G(u_1')\backslash \{u_2'\}$. Since $G$ is matching covered, there is a perfect matching $M$ that contains $u_1'w$. Recall that $\partial(X_1)$ is a tight cut of $G$. Then $M\cap \partial(X_1)=\{u_1'w\}$. So $u_1''u_2''\in M$, that is, $u_1'w$ depends on $u_1''u_2''$ in $G$. By Lemma \ref{depend-on}, $u_1''u_2''$ is nonremovable in $G$. Likewise, $u_1'u_2'$ is also nonremovable in $G$. Therefore, $$RE(G)=RE(G)\cap (E(G_1)\cup E(G_2))=RE(G_1)\cup RE(G_2).$$
(ii) (a) follows. 
\end{proof}

\section{Proof of Theorem \ref{claw-free-mc}}

Recall that a matching covered graph $G$ is minimal if and only if $RE(G)=\emptyset$.

\begin{lemma}\label{clawfree-mc-3}
A graph $G$ is a  claw-free minimal matching covered graph with $\delta(G)\geq3$ if and only if $G\in\mathcal{F}$.
\end{lemma}

\begin{proof}
Suppose that $G\in\mathcal{F}$. By induction on $|V(G)|$. If $G\in\mathcal{G}$, the result holds by Theorem \ref{claw-free-bicriticalR(G)emptyset}. Now assume that $G$ is a compound of  two graphs $G_1$ and $G_2$ in $\mathcal{F}$. By the induction hypothesis, $G_i$ is a claw-free minimal matching covered graph with $\delta(G_i)\geq3$, $i=1,2$. Therefore, $\delta(G)\geq3$ and $RE(G_1)=RE(G_2)=\emptyset$. By Lemma \ref{four-operation-property}, $G$ is a claw-free matching covered graph with $RE(G)=\emptyset$, that is, $G$ is a claw-free minimal matching covered graph.

Conversely, suppose that $G$ is a claw-free minimal matching covered graph with $\delta(G)\geq3$. Then $RE(G)=\emptyset$, so $G$ is a simple graph. We show, by induction on $|V(G)|$, that $G\in\mathcal{F}$.
If $G$ is bicritical, Theorem \ref{claw-free-bicriticalR(G)emptyset} implies that $G\in \mathcal{G}$. The assertion holds.

Now suppose that  $G$ is not bicritical. By Lemma \ref{clawfreebicritical-no2barrier}, $G$ has a 2-barrier $\{u,v\}$.
Since $G$ is matching covered, by Proposition \ref{mc-B-stableset}, we have $uv\notin E(G)$, and $G-\{u,v\}$ has no even components. Then $G-\{u,v\}$ has only two odd components, say $L_1$ and $L_2$.
Since $G$ is 2-connected, each of $u$ and $v$ has at least one neighbour in $L_1$ and  at least one neighbour in $L_2$.
We first show the following claim.
\vspace{1.5mm}

\textbf{Claim 1.} $N(u)\cap V(L_i)$ and $N(v)\cap V(L_i)$ are two disjoint cliques, $i=1,2$.

Since $G$ is claw-free, $N(u)\cap V(L_i)$ and $N(v)\cap V(L_i)$ are two cliques, $i=1,2$. We  show that these two cliques are disjoint.
Suppose, to the contrary, that $w\in N(u)\cap N(v)\cap V(L_i)$.
Then $|V(L_i)|\geq3$ because $\delta(G)\geq3$, $G$ is simple and $L_i$ is odd.
Since $G$ is 2-connected, one of $u$ and $v$, say $u$, has a neighbour in $L_i$ except $w$. Observe that $uw$ is a multiple edge of each of $G\{V(L_i)\}$ and $G\{\overline{V(L_i})\}$, so is a removable edge of each of them. Note  that $\partial(V(L_i))$ is a barrier cut of $G$, so is a tight cut of $G$. Lemma \ref{removable-tight} implies that  $uw$ is removable in $G$,  contradicting the assumption that $RE(G)=\emptyset$.
This proves Claim 1.
\vspace{1.5mm}

Since $RE(G)=\emptyset$, Lemma \ref{K5freeK4}(i) implies that $G$ is $K_5$-free. Thus each of $u$ and $v$ has at most three neighbours in $L_1$ and $L_2$, respectively. Therefore, $2\leq|\partial(V(L_i))|\leq6$, $i=1,2$. Because $d(u)\geq3$ and $d(v)\geq3$, at most one of $|\partial(V(L_1))|$ and $|\partial(V(L_2))|$ is 2.
According to the cardinality of  $\partial(V(L_1))$ and $\partial(V(L_2))$,  we consider the following two cases.
In each case, we will find two graphs $G_1$ and $G_2$ that satisfy the following  property:
$G_i$ is a claw-free minimal matching covered graph with $\delta(G_i)\geq3$, $i=1,2$.
\vspace{1.5mm}

\textbf{Case 1.} $|\partial(V(L_i))|=2$, $i=1$ or 2.

Suppose, without loss of generality, that  $|\partial(V(L_1))|=2$.
Let $x_1$ and $y_1$ be the neighbour of $u$ and $v$ in $L_1$, respectively. By Claim 1, $x_1\neq y_1$, and $N(u)\cap V(L_2)$ and $N(v)\cap V(L_2)$ are two disjoint cliques. Since $\delta(G)\geq3$, we have $|N(u)\cap V(L_2)|\geq2$ and $|N(v)\cap V(L_2)|\geq2$. Since $L_1$ is odd, $L_1-\{x_1,y_1\}$ has at least one vertex. Therefore, $\{x_1,y_1\}$ is a 2-vertex cut of $G$, and $G-\{x_1,y_1\}$ has an odd component $G[V(L_2)\cup \{u,v\}]$. By Proposition \ref{2-barrier}, $\{x_1,y_1\}$ is a 2-barrier of $G$.
Thus $L_1-\{x_1,y_1\}$ is the other odd component of $G-\{x_1,y_1\}$, and $x_1y_1\notin E(G)$ by Proposition \ref{mc-B-stableset}.
From Claim 1, we see that $N(x_1)\cap V(L_1)$ and $N(y_1)\cap V(L_1)$ are  two disjoint cliques. Moreover, since $\delta(G)\geq3$, we have $|N(x_1)\cap V(L_1)|\geq2$ and $|N(y_1)\cap V(L_1)|\geq2$.
Let $G_1=G\{V(L_1)\backslash\{x_1,y_1\},w_1\}$ and $G_2=G\{V(L_2),w_2\}$. Then $w_i$ is a bisimplicial vertex of $G_i$ and $d_{G_i}(w_i)\geq4$, $i=1,2$. Thus $G$ can be obtained from $G_1$ and $G_2$ by V-joining at $w_1$ and $w_2$. By Lemma \ref{four-operation-property},
$G_i$ is a claw-free matching covered graph with $RE(G_i)=\emptyset$, that is, $G_i$ is a claw-free minimal matching covered graph, $i=1,2$.
Since $\delta(G)\geq3$, we have $\delta(G_1)\geq3$ and $\delta(G_2)\geq3$. Consequently,   $G_1$ and $G_2$ satisfy the required property.
\vspace{1.5mm}

\textbf{Case 2.} $|\partial(V(L_i))|\geq3$, $i=1,2$.

Let $G_i=G\{V(L_i),w_i\}$. Then $\delta(G_i)\geq3$. Recall that $\partial(V(L_i))$ is a barrier cut of $G$.
Then the graph $G_i$ is matching covered. By Claim 1, $N(u)\cap V(L_i)$ and $N(v)\cap V(L_i)$ are  two disjoint cliques, so $G_i$ has no claw with centre $w_i$.
Since $G$ is claw-free, so does $G_i$. 
If each of $u$ and $v$ has at least two neighbours in $L_i$, then each edge of $\partial(V(L_i))$ is a multiple edge of $G\{\overline{V(L_i})\}$, so is removable in $G\{\overline{V(L_i})\}$.
Since $RE(G)=\emptyset$, by Lemma \ref{removable-tight},  we have $RE(G_i)=\emptyset$.
Now suppose, without loss of generality, that $u$ has exactly one neighbour $x_i$ in $L_i$.
Since $|\partial(V(L_i))|\geq3$, $v$ has at least two neighbours in $L_i$.
Recall that $v$ has at most three neighbours in $L_i$. Then $v$ has either two or three neighbours in $L_i$.
Note that all the edges of $\partial(V(L_i))$ incident with $v$  are multiple edges in $G\{\overline{V(L_i})\}$. Again by Lemma \ref{removable-tight} and the fact that $RE(G)=\emptyset$, we have $RE(G_i)\subseteq\{w_ix_i\}$. However, Lemma \ref{K5freeK4}(ii) implies that $w_ix_i$ is nonremovable in $G_i$. It follows that $RE(G_i)=\emptyset$. Consequently, $G_i$ satisfies the required property.
\vspace{1.5mm}

By the induction hypothesis, both $G_1$ and $G_2$ belong to $\mathcal{F}$. By Lemma \ref{4-degree-2K2}, each of $G_1$ and $G_2$ is either isomorphic to $K_4$ or $K_4$-free. Moreover, for any vertex $x^{*}\in V(G_i)$, we have $d_{G_i}(x^{*})=3$ or 4, and $G_i[N_{G_i}(x^{*})]$ is the union  of two disjoint $K_2$s if $d_{G_i}(x^{*})=4$, $i=1,2$.
Therefore, each of $u$ and $v$ has at most two neighbours in $L_1$ and $L_2$, respectively, so has degree  3 or 4 in $G$.

We proceed to show that $G\in \mathcal{F}$. In Case 1, $G$ is obtained from $G_1$ and $G_2$ by V-joining. We are done.
Now we consider Case 2. If $d(u)=d(v)=4$, then each of $u$ and $v$ has exactly two neighbours in $L_1$ and $L_2$, respectively. So $d_{G_i}(w_i)=4$ and $w_i$ is a bisimplicial vertex of $G_i$, $i=1,2$. Consequently,  $G$ can be obtained from $G_1$ and $G_2$ by V-attaching at $w_1$ and $w_2$. Then $G\in \mathcal{F}$.
Now, suppose, without loss of generality, that $d(v)=3$ and $v$ has exactly one neighbour, say $v_1$, in $L_1$. Then $|\partial(V(L_1))|=3$ and $u$ has exactly two neighbours in $L_1$, say $\{u_1,u_2\}$.
By Claim 1, $u_1u_2\in E(G)$. Recall that $G_i=G\{V(L_i),w_i\}$, $i=1,2$. We  prove the following claim.
\vspace{1.5mm}

\textbf{Claim 2.} $w_1v_1$ is a bisimplicial edge of $G_1$. 

Recall that each edge of $K_4$ is a bisimplicial edge. The result holds when $G_1$ is isomorphic to $K_4$. Now suppose that $G_1$ is not isomorphic to $K_4$. Then $G_1$ is $K_4$-free.
As $G$ is claw-free, $N_G(v_1)\cap V(L_1)$ is a clique. Since $G_1$ is $K_4$-free and $\delta(G)\geq3$, we have $|N_G(v_1)\cap V(L_1)|=2$. Let $N_G(v_1)\cap V(L_1)=\{v_{11},v_{12}\}$.
Since $G$ is claw-free, we have $v_{11}v_{12}\in E(G)$. If $\{v_{11},v_{12}\}=\{u_1,u_2\}$, then $G_1$ contains $K_4$ as a subgraph induced by $\{w_1,u_1,u_2,v_1\}$, a contradiction. 
Thus $\{v_{11},v_{12}\}\neq\{u_1,u_2\}$. Assume, without loss of generality, that $u_2=v_{12}$ and $u_1\neq v_{11}$. Then  $\{u_1,w_1,v_1,v_{11}\}\subseteq N_{G_1}(u_2)$. Recall that $d_{G_1}(u_2)=3$ or 4. We have  $d_{G_1}(u_2)=4$. Thus $G_1[N_{G_1}(u_2)]$ is the union of two  $K_2$s. But this contradicts the fact that
$G_1[N_{G_1}(u_2)]$ is connected.
Therefore, $\{v_{11},v_{12}\}\cap\{u_1,u_2\}=\emptyset$, that is, $w_1v_1$ is a bisimplicial edge of $G_1$. This proves  Claim 2.
\vspace{1.5mm}

If $d(u)=4$, by Claim 1,  $w_2$ is a bisimplicial vertex of $G_2$. Then $G$ can be obtained from $G_1$ and $G_2$ by EV-attaching at $w_1v_1$ and $w_2$.  Thus $G\in\mathcal{F}$.

If $d(u)=3$,   let $u_3$ be the only neighbour of $u$ in $L_2$. By a similar argument as that in the proof of  Claim 2, we can show that $w_2u_3$ is a bisimplicial edge of $G_2$. Then $G$ can be obtained from $G_1$ and $G_2$ by E-joining at $w_1v_1$ and $w_2u_3$. Thus $G\in\mathcal{F}$. This completes the proof of Lemma \ref{clawfree-mc-3}.
\end{proof}

\begin{lemma}\label{bisubdivision}
Suppose that $G$ is a bisubdivision of a graph $H$ different from $K_2$ at an edge $e$.  Then
the following statements hold.
\vspace{-8pt}
\begin{enumerate}[(\romannumeral1)]
\setlength{\itemsep}{-1ex}
\item  If $H$ is claw-free and matching covered,  and $e$ is a ridge of $H$, then $G$ is  claw-free and matching covered.
\item  If $G$ is  claw-free and matching covered, then $H$ is claw-free and matching covered, and $RE(G)=RE(H)\backslash\{e\}$.
\end{enumerate}
\end{lemma}

\begin{proof}
Since $G$ is a bisubdivision  of $H$ at the edge $e$, $G$ is obtained from $H$ by replacing $e$ by an odd path $P_e$ with length at least three. Let $e=ab$ and $X=V(P_e)\backslash\{b\}$. Then $H$ is isomorphic to $G\{\overline{X}\}$. Note that $G$ is a splicing of $G\{\overline{X}\}$ and $G\{X\}$. We are now ready to prove (i) and (ii).

(i) Assume that $H$ is claw-free and matching covered, and $e$ is a ridge of $H$.
Since $H$ is different from $K_2$, $H$ is 2-connected. Then $\delta(H)\geq2$. Therefore, the underlying simple graph of $G\{X\}$ is an even cycle. So $G\{X\}$ is matching covered.
By Lemma \ref{splicing}, $G$ is matching covered. Recall that each edge of $K_4$ is a ridge. If $H$ is isomorphic to $K_4$, then $G$ is claw-free. Now suppose that $H$ is not isomorphic to $K_4$.  Since  $ab$ is a ridge of $H$,  $a$ and $b$ have no neighbours in common. Since $H$ is claw-free, $N_H(a)\backslash\{b\}$ and $N_H(b)\backslash\{a\}$ are two disjoint cliques. This implies that $G$ has no claws with centres $a$ and $b$. Thus $G$ is claw-free.

(ii) Assume that $G$ is claw-free and matching covered.
Then $H$ is also claw-free. Since $P_{e}-b$ is an   even path of $G$ with internal vertices of degree 2, for each perfect matching $M$ of $G$, we have $|M\cap \partial_G(X)|=1$. So $\partial_G(X)$ is a tight cut of $G$.
It follows that $H$ is matching covered.

Note that each edge of $P_e$ is incident with a vertex of degree 2, so is nonremovable in $G$.
Therefore, if $f\in RE(G)$, then $f\in E(H)\backslash\{e\}$. By Lemma \ref{removable-tight}, we have $f\in RE(H)\backslash\{e\}$. So $RE(G)\subseteq RE(H)\backslash\{e\}$. Now assume that $f\in RE(H)\backslash\{e\}$. If $f\notin \partial(X)$, Lemma \ref{removable-tight} implies that $f\in RE(G)$. If $f\in \partial(X)$, then $f$ is incident with $a$ in $H$. Since $f$ is removable in $H$, we have  $d_H(a)\geq3$. So $d_G(a)=d_H(a)\geq3$. It follows that $f$ is a multiple edge of $G\{X\}$, so is a removable edge of $G\{X\}$. Again by Lemma \ref{removable-tight}, we have $f\in RE(G)$. It follows that $RE(H)\backslash\{e\}\subseteq RE(G)$. Consequently, $RE(G)=RE(H)\backslash\{e\}$.
\end{proof}

\begin{lemma}\label{vertex-expansion}
Suppose that $G$ is an expansion of a graph $H$ at  a bisimplicial vertex. Then the following statements hold.
\vspace{-8pt}
\begin{enumerate}[(\romannumeral1)]
\setlength{\itemsep}{-1ex}
\item  $G$ is claw-free and matching covered if and only if $H$ is claw-free and matching covered.
\item  If $G$ is claw-free and matching covered, then $RE(G)=RE(H)$.
\end{enumerate}
\end{lemma}

\begin{proof}
Suppose that $G$ is an expansion of $H$ at  a bisimplicial vertex $x$, and $x'$ and $x''$ are the two split vertices. Then $N_H(x)$ is the union of two disjoint cliques such that each of them has at least two vertices,  $x'$ is adjacent to each vertex of one clique and $x''$ is adjacent to each vertex of the other clique, and the $x'x''$-path $P_x$ is an even path with length at least two.
Let $X=V(P_x)$. Then $H$ is isomorphic to $G\{\overline{X}\}$, and the underlying simple graph of $G\{X\}$ is an even cycle. So $G\{X\}$ is matching covered.

Assume that $H$ is claw-free and matching covered. By Lemma \ref{splicing}, $G$ is matching covered. Since $H$ is claw-free, so does $G$.
Conversely, assume that $G$ is claw-free and matching covered.
Then $H$ is also claw-free. Note that, for each perfect matching $M$ of $G$, we have $|M\cap \partial_G(X)|=1$, so $\partial_G(X)$ is a tight cut of $G$.
Then $H$ is matching covered. Since each edge $f$ of $P_x$ is incident with a vertex of degree 2, $f$ is nonremovable in $G$. Thus $RE(G)\subseteq RE(H)$.  Each edge of $\partial_G(X)$ is a multiple edge of $G\{X\}$, so is a removable edge of $G\{X\}$. By Lemma \ref{removable-tight}, we have $RE(H)\subseteq RE(G)$. Consequently, $RE(G)=RE(H)$.
\end{proof}

\begin{lemma}\label{expansion}
If $G$ is an expansion of a graph in $\mathcal{F}$, then $G$ is a claw-free minimal matching covered graph. 
\end{lemma}

\begin{proof}
By induction on $|V(G)|$. Suppose that $G$ is an expansion of a graph $H\in\mathcal{F}$. Then there is a sequence of graphs $(G_1,G_2,\ldots,G_{r+1})$ such that (i) $G_1=H$ and $G_{r+1}=G$, (ii) $G_{i+1}$ is either a bisubdivision of $G_i$ at a ridge or an expansion of $G_i$ at a bisimplicial vertex, $1\leq i\leq r$. If $r=0$, then $G=H$. By Lemma \ref{clawfree-mc-3}, $G$ is a claw-free minimal matching covered graph.
Now assume that $r\geq1$. Note that $G_i$ is an expansion of $H$, $2\leq i\leq r+1$.
By the induction hypothesis, $G_r$ is a claw-free minimal matching covered graph, so $RE(G_r)=\emptyset$. Using Lemmas \ref{bisubdivision} and \ref{vertex-expansion}, $G$ is a claw-free matching covered graph with $RE(G)=\emptyset$,
that is, $G$ is a claw-free minimal matching covered graph.
\end{proof}

{\it Proof of Theorem \ref{claw-free-mc}.}
Suppose that $G$ is an expansion of a graph in $\mathcal{F}$. By Lemma \ref{expansion}, $G$ is a claw-free minimal matching covered graph. The result holds.

Conversely, assume that $G$ is a claw-free minimal matching covered graph, and it is different from $K_2$ and even cycle.  Then $RE(G)=\emptyset$.
By induction on $|V(G)|$. If $\delta(G)\geq3$, Lemma \ref{clawfree-mc-3} implies that $G\in\mathcal{F}$. The result holds.
Since $G$ is different from $K_2$,  we may assume that $G$ has a vertex of degree 2.
Let $P=u_1u_2\ldots u_k$ be a maximal  path of $G$ with internal vertices of degree 2 in $G$.
Then $k\geq3$ and $|E(P)|\geq2$.
Since $G$ is different from even cycle, both $u_1$ and $u_k$ have degree at least 3 in $G$.

We first assert that $u_1u_k\notin E(G)$.  Suppose, to the contrary, that $u_1u_k\in E(G)$.  If $P$ is even, then $u_1u_k$ is inadmissible in $G$, contradicting the assumption that $G$ is matching covered.
If $P$ is odd,  let $X=V(P)\backslash \{u_k\}$.
Then, for each perfect matching $M$ of $G$, we have $|M\cap\partial_G(X)|=1$, so $\partial_G(X)$ is a tight cut of $G$. Since $d_G(u_1)\geq 3$, $u_1u_k$ is a multiple edge of the two graphs $G\{X\}$ and $G\{\overline{X}\}$, so is removable in each of them. By Lemma \ref{removable-tight}, $u_1u_k$ is removable in $G$, contradicting the fact that $RE(G)=\emptyset$. The assertion holds.

Let $H$ be the graph obtained from $G-(V(P)\backslash \{u_1,u_k\})$ by adding a new edge $e^{*}$ that connects $u_1$ and $u_k$ if $P$ is odd, and by identifying $u_1$ and $u_k$ into a new vertex $x^{*}$ otherwise.
Then $H$ is distinct from $K_2$ and even cycle.
Let $N_1=N_G(u_1)\cap (V(G)\backslash V(P))$ and $N_2=N_G(u_k)\cap (V(G)\backslash V(P))$.
Since $G$ is claw-free, $d_G(u_1)\geq 3$, $d_G(u_k)\geq 3$ and $u_1u_k\notin E(G)$,  we see that  both $N_1$ and $N_2$ are two cliques, and each of them has at least two vertices. We proceed the proof by distinguishing whether $P$ is even or not.

If $P$ is even, then $N_1\cap N_2=\emptyset$. Otherwise, assume that $w\in N_1\cap N_2$.
Let $Y=V(P)$. Then, for each perfect matching $M$ of $G$, we have $|M\cap\partial_G(Y)|=1$, so $\partial_G(Y)$ is a tight cut  of $G$.
Note that $wu_1$ is a multiple edge of $G\{Y\}$ and $G\{\overline{Y}\}$, so is  removable in each of them.
By Lemma \ref{removable-tight}, $wu_1$ is removable in $G$, contradicting the fact that $RE(G)=\emptyset$.
Thus $x^{*}$ is a bisimplicial vertex of $H$, so $G$ is an expansion of $H$ at the bisimplicial vertex $x^{*}$.
By Lemma \ref{vertex-expansion}, $H$ is a claw-free matching covered graph with $RE(H)=\emptyset$, that is, $H$ is a claw-free minimal matching covered graph.
By  induction hypothesis, $H$ is an expansion of a graph in $\mathcal{F}$. Consequently, $G$ is also an expansion of a graph in $\mathcal{F}$.

If $P$ is odd, then $G$ is a bisubdivision of $H$ at the edge $e^{*}$. Note that $d_H(u_1)\geq 3$ and $d_H(u_k)\geq 3$.
By Lemma \ref{bisubdivision}, $H$ is a claw-free matching covered graph with $RE(H)\subseteq \{e^{*}\}$.
Let $Q_1=(N_H(u_1)\backslash\{u_k\})\cup \{u_1\}$ and $Q_2=(N_H(u_k)\backslash\{u_1\})\cup \{u_k\}$.
Then $|Q_i|\geq3$ and $Q_i$ is a clique of $G$ and $H$, respectively, $i=1,2$. Since $RE(G)=\emptyset$, Lemma \ref{K5freeK4}(i) implies that $|Q_i|=3$ or 4, $i=1,2$.
Because $u_1$ has exactly one neighbour $u_k$ in $V(H)\backslash Q_1$, by Lemma \ref{K5freeK4}(ii), $e^{*}$ is nonremovable in $H$. Therefore, $RE(H)=\emptyset$.
Then $H$ is a claw-free minimal matching covered graph.
By the induction hypothesis, $H$ is an expansion of a graph in $\mathcal{F}$.

We now show that $e^{*}$ is a ridge of $H$. If this is true, then $G$ is an expansion of a graph in $\mathcal{F}$.
Suppose, to the contrary, that $e^{*}$ is not a ridge of $H$. Recall that each edge of $K_4$ is a ridge. Then $H$ is not isomorphic to $K_4$, and $u_1$ and $u_k$ have a common neighbour in $H$, say $z$. By Lemma \ref{4-degree-2K2}, $H$ is $K_4$-free.
Since $Q_i$ is a clique in $H$, we have $|Q_i|=3$, $i=1,2$. Thus $d_H(u_1)=3$ and $d_H(u_k)=3$.
Let $N_H(u_1)=\{u_k,z,v_1\}$ and $N_H(u_k)=\{u_1,z,v_k\}$.
Then $v_1z,v_kz\in E(H)$ because both $Q_1$ and $Q_2$ are cliques in $H$.
If $v_1\neq v_k$, then $\{u_1,u_k,v_1,v_k\}\subseteq N_H(z)$. By Lemma \ref{4-degree-2K2} again, we see that $d_H(z)=4$ and $H[N_H(z)]$ is the union  of two disjoint $K_2$s. But this contradicts the fact that $H[N_H(z)]$ is connected.  So  $v_1=v_k$. Then $H[\{u_1,u_k,v_1,z\}]$ is isomorphic to $K_4$, contradicting the fact that $H$ is $K_4$-free.
This completes the proof of Theorem \ref{claw-free-mc}.

\section{Proof of Theorem \ref{2-connected-cubic-claw-free}}

Recall that  each  multiple edge of a matching covered graph is a removable edge. In this Section we only consider simple graphs. Note that for a cubic claw-free graph $G$, a ridge of $G$ is exactly a bisimplicial edge of $G$, and vice versa. We first prove the following lemma.

\begin{lemma}\label{cubic-ridge-nonremovable}
Let $G$ be a cubic claw-free matching covered graph. Then each ridge of $G$ is nonremovable in $G$.
\end{lemma}

\begin{proof}
Suppose that $ab$ is a ridge of $G$. Since $G$ is cubic, vertex $a$ has exactly two neighbours $a_1$ and $a_2$ other than $b$. Because $G$ is claw-free, we have $a_1a_2\in E(G)$. It follows that $T=aa_1a_2a$ is a triangle and the vertex $a$ has exactly one neighbour $b$ in $V(G)\backslash V(T)$. By Lemma \ref{K5freeK4}(ii), $ab$ is nonremovable in $G$.
\end{proof}

Kothari et al. \cite{KCLL2020} showed that for a bicritical graph $G$, if $G$ is not a brick, then $G$ has at least two non-cubic vertices. The following is an easy application of this result.

\begin{lemma}\label{cubic-bicritical-brick}
Every cubic bicritical graph is a brick.
\end{lemma}

The following two lemmas are easily verified.

\begin{lemma}\label{3-connected-3-cut-matching}
In a 3-connected graph, every nontrivial 3-cut is a matching.
\end{lemma}

\begin{lemma}\label{cubic-cut}
Let $G$ be a 2-connected cubic graph and let $S\subseteq V(G)$. If $|S|$ is odd, then $|\partial(S)|\geq 3$.
\end{lemma}

By Lemma \ref{Tutte}, we can show that every 2-edge-connected cubic graph is  matching covered. Note that a cubic matching covered graph is 2-edge-connected. Thus a cubic graph is matching covered if and only if it is 2-edge-connected. It is known that a cubic graph is 2-connected if and only if it is 2-edge-connected. This implies that a cubic graph is matching covered if and only if it is 2-connected.
To obtain the number of removable edges of cubic claw-free matching covered graphs, we first consider 3-connected graphs.

\begin{lemma}\label{3-connected-cubic-claw-free}
Let $G$ be a 3-connected cubic claw-free graph different from $K_4$ and $\overline{C_6}$. Then $RE(G)$ is the set of all the edges of $G$ that lie in a triangle. Consequently, $G$ has exactly $|V(G)|$ removable edges.
\end{lemma}

\begin{proof}
Note that $G$ is matching covered. Since $G$ is a cubic claw-free graph different from $K_4$,  each vertex of $G$ lies in one or two triangles.
If there exists a vertex of $G$ that lies in exactly two triangles,
since $G$ is cubic, $G$ has a 2-vertex cut, contradicting the fact that $G$ is 3-connected. Therefore, each vertex of $G$ lies in exactly one triangle of $G$.
\vspace{1.5mm}

\textbf{Claim.}  Each edge that lies in a triangle of $G$ is removable in $G$.

We prove this claim by contradiction.  Let $T=xyzx$ be a triangle of $G$. Suppose that an edge of $T$, say $xy$, is nonremovable in $G$. Then $G-xy$ is not matching covered, so it has an inadmissible edge $f$. By Proposition \ref{prop2}(i), $G-xy$ has a barrier that contains the two ends of $f$. Let $B$ be such a maximal barrier. By Proposition \ref{prop2}(ii), each component of $G-xy-B$ is factor-critical. Since $f$ is admissible in $G$, $xy$ has its ends in distinct components of $G-xy-B$. Let $L_1$ and $L_2$ be the two components of $G-xy-B$ such that $x\in V(L_1)$ and $y\in V(L_2)$.
Because $G$ is cubic and $f$ has its ends in $B$, we have $|\partial_G(B)|\leq3|B|-2$. Also, for each component $J$ of $G-xy-B$, by Lemma \ref{cubic-cut}, we see that $|\partial_G(J)|\geq3$. Then $|\partial_G(B)|\geq3|B|-2$, so $|\partial_G(B)|=3|B|-2$. This implies that $E(G[B])=\{f\}$  and $|\partial(J)|=3$ for each component $J$ of $G-xy-B$. Since $xz,yz\in E(G)$, we have $z\in B$. By Lemma \ref{3-connected-3-cut-matching}, both $L_1$ and $L_2$ are trivial. According to whether $z$ is an end of $f$ or not, we consider the following two cases.

First suppose that $z$ is an end of $f$. If $|B|\geq3$, there exists a vertex $u$ in $B$ such that $u$ is not an end of $f$. Note that the vertex $u$ lies in a triangle of $G$, say $uvwu$. Then either $vw=xy$ or $vw$ lies in a component of $G-xy-B$ because $E(G[B])=\{f\}$.
For the former case, $x$ lies in two triangles $xyz$ and $xyu$ of $G$, a contradiction.
For the latter case, let $L$ be the  component of $G-xy-B$ that contains $vw$. Then $\partial_G(V(L))$ is not a matching, contradicting Lemma \ref{3-connected-3-cut-matching}. Therefore, $|B|=2$. This implies that $G$ is $K_4$, a contradiction.

Now suppose that $z$ is not an end of $f$. Let $f=u'v'$. Applying a similar argument as in the above case, we can show that $B$ exactly contains three vertices, $z,u'$, and $v'$. Then $G-xy-B$ has exactly one component, say $L_3$, other than $L_1$ and $L_2$.
Since each vertex of $G$ lies in exactly one triangle of $G$, $x$ and $y$ have no common neighbours other than $z$.
Assume, without loss of generality, that $xu',yv'\in E(G)$. Since $G$ is cubic, $u'$ has a neighbour in $L_3$, say $w'$. Then $v'w'\in E(G)$ because $G$ is claw-free. By Lemma \ref{3-connected-3-cut-matching}, $L_3$ is trivial. So $zw'\in E(H)$. This implies that $G$ is $\overline{C_6}$, a contradiction. The claim holds.
\vspace{1.5mm}

By Lemma \ref{cubic-ridge-nonremovable},  each ridge of $G$ is nonremovable in $G$. By the above Claim,  $RE(G)$ is the set of all the edges of $G$ that lie in a triangle.
Since each vertex of $G$ lies in exactly one triangle of $G$, we have $|RE(G)|=\frac{|V(G)|}{3}\times3=|V(G)|$. The result holds.
\end{proof}

Plummer \cite{Plummer94} showed that every 3-connected claw-free graph with an even number of vertices is a brick. Recall that a graph is a brick if and only if it is 3-connected and bicritical. It follows that every 3-connected cubic claw-free graph is a cubic claw-free brick, and vice versa.
From Lemma \ref{3-connected-cubic-claw-free}, we see that
every cubic claw-free brick $G$ different from $K_4$ and $\overline{C_6 }$ has exactly $|V(G)|$ removable edges.

\vspace{2.5mm}
{\it Proof of Theorem \ref{2-connected-cubic-claw-free}.}
Let $G$ be a cubic claw-free matching covered graph.
Recall that $b^{*}(G)$ denotes the number of the bricks of $G$ each of which is different from $K_4$ and $\overline{C_6}$.
If $b^{*}(G)\geq1$, the number of the vertices of the $b^{*}(G)$ bricks is denoted by $n_1, n_2,\ldots, n_{b^{*}(G)}$, respectively. To complete the proof of Theorem \ref{2-connected-cubic-claw-free}, we first prove the following claim.
\vspace{1.5mm}

\textbf{Claim.}  If $G$ is minimal matching covered, then $b^{*}(G)=0$; otherwise, $b^{*}(G)\geq1$ and $|RE(G)|=\sum_{i=1}^{b^{*}(G)}n_i$.

We shall prove the claim by induction on $|V(G)|$.
Assume first that $G$ is a brick. Then $G$ is 3-connected. Recall that Lov\'asz \cite{Lovasz87} showed that the only bricks that are minimal matching covered are $K_4$ and $\overline{C_6}$.
If $G$ is minimal matching covered, then $G$ is  $K_4$ or $\overline{C_6}$. So $b^{*}(G)=0$. If $G$ is not minimal matching covered, then $G$ is different from $K_4$ and $\overline{C_6}$. So $b^{*}(G)=1$ and $n_1=|V(G)|$.
By Lemma \ref{3-connected-cubic-claw-free}, we have $|RE(G)|=|V(G)|=n_1$. So the result holds when $G$ is a brick.

Now assume that $G$  is not a brick. Since $G$ is cubic, by Lemma \ref{cubic-bicritical-brick}, $G$ is not bicritical. By Lemma \ref{clawfreebicritical-no2barrier}, $G$ has a 2-barrier $\{u,v\}$. By Proposition \ref{mc-B-stableset}, we have $uv\notin E(G)$ and $G-\{u,v\}$ has no even components. It follows that $G-\{u,v\}$ has exactly two odd components, say $L_1$ and $L_2$.
Moreover, we have $|V(L_1)|\geq3$ and $|V(L_2)|\geq3$ because $G$ is simple and cubic.
Note that $G$ is 2-connected. Then each of $u$ and $v$ has at least one neighbour in $L_1$ and at least one neighbour in $L_2$. By Lemma \ref{cubic-cut}, we have $|\partial(V(L_1))|\geq3$ and $|\partial(V(L_2))|\geq3$. On the other hand, since $G$ is cubic, we have $|\partial(V(L_1))|+|\partial(V(L_2))|=d(u)+d(v)=6$. So $|\partial(V(L_1))|=|\partial(V(L_2))|=3$.
Assume, without loss of generality, that $u$ has exactly two neighbours $u_1$ and $u_2$ in $L_1$, and $v$ has exactly one neighbour $v_3$ in $L_1$.
Then $u$ has exactly one neighbour $u_3$ in $L_2$, and $v$ has exactly two neighbours  in $L_2$.
Since $G$ is claw-free, we have $u_1u_2\in E(G)$.
If $v_3$ is one of $u_1$ and $u_2$, say $u_2$, then $u_1$ is a cut vertex of $G$, a contradiction. So $v_3\notin\{u_1,u_2\}$.
Let $G_i=G\{V(L_i),w_i\}$, $i=1,2$.

Next we shall prove that $RE(G)=RE(G_1)\cup RE(G_2)$.
We first assert that $w_1v_3$ is a ridge of $G_1$.
Since $G$ is cubic, $v_3$ has exactly two neighbours $v_{31}$ and $v_{32}$ in $L_1$ other than $v$. Then $v_{31}v_{32}\in E(G)$ because $G$ is claw-free. We see that $|\{u_1,u_2\}\cap\{v_{31},v_{32}\}|\neq1$ because $G$ is cubic.
If $\{u_1,u_2\}\cap\{v_{31},v_{32}\}=\emptyset$, then $w_1v_3$ is a ridge of $G_1$.
If $\{u_1,u_2\}=\{v_{31},v_{32}\}$, then $G_1$ is isomorphic to $K_4$. Recall that each edge of $K_4$ is a ridge. The assertion holds.
By symmetry, we can show that $w_2u_3$ is a ridge of $G_2$. Recall that for a cubic claw-free graph, a ridge is a bisimplicial edge. Then $G$ is isomorphic to $(G_1\oplus G_2)_{(w_1v_3,u_3w_2)}$. Since $G$ is cubic, both $G_1$ and $G_2$ are cubic. By Lemma \ref{four-operation-property}, both $G_1$ and $G_2$ are claw-free matching covered graphs, and $RE(G)=(RE(G_1)\cup RE(G_2))\backslash\{w_1v_3,w_2u_3\}$.
Moreover, by Lemma \ref{cubic-ridge-nonremovable},
$w_1v_3$ is nonremovable in $G_1$ and $w_2u_3$ is nonremovable in $G_2$.
Therefore, we have $RE(G)=RE(G_1)\cup RE(G_2)$.

Note that $\partial_G(V(L_i))$ is a barrier cut of $G$, $i=1,2$.
Then $G_1$ and $G_2$ can arise during the procedure of a tight cut decomposition of $G$.
Moreover, the union of the lists of bricks of $G_1$ and $G_2$ is the same as the list of bricks of $G$, using the fact that any two tight cut decompositions of a matching covered graph yield the same list of bricks and braces. In particular, we have $b^{*}(G)=b^{*}(G_1)+b^{*}(G_2)$.
If $G$ is minimal matching covered, then $RE(G)=\emptyset$.
So $RE(G_1)=RE(G_2)=\emptyset$, that is, both $G_1$ and $G_2$ are minimal matching covered.
By the induction hypothesis, we have $b^{*}(G_1)=b^{*}(G_2)=0$. So $b^{*}(G)=b^{*}(G_1)+b^{*}(G_2)=0$.
If $G$ is not minimal matching covered, then $RE(G)\neq\emptyset$. Since $RE(G)=RE(G_1)\cup RE(G_2)$, one of $G_1$ and $G_2$, say $G_1$, has removable edges. So $G_1$ is not minimal matching covered. If $G_2$ is not minimal matching covered, by the induction hypothesis, we have
$|RE(G_1)|+|RE(G_2)|=\sum_{i=1}^{b^{*}(G)}n_i$. So $|RE(G)|=|RE(G_1)|+|RE(G_2)|=\sum_{i=1}^{b^{*}(G)}n_i$.
If $G_2$ is minimal matching covered, then $RE(G_2)=\emptyset$. Moreover, by the induction hypothesis, we have $b^{*}(G_2)=0$.
Again by the induction hypothesis, we have $|RE(G_1)|+|RE(G_2)|=|RE(G_1)|=\sum_{i=1}^{b^{*}(G)}n_i$. So $|RE(G)|=|RE(G_1)|+|RE(G_2)|=\sum_{i=1}^{b^{*}(G)}n_i$.
This proves the claim.
\vspace{1.5mm}

From the proof of the above Claim, we see that every brick of $G$ is cubic claw-free. Now assume that $G$ is not minimal matching covered.

(i) By the above Claim, (i) holds.

To show (ii), we first assert that every cubic claw-free brick $H$ different from $K_4$ and $\overline{C_6 }$ has at least 12 vertices.
Since $H$ is 3-connected and distinct from $K_4$,
each vertex of $H$ lies in exactly one triangle of $H$.
Consequently,  $|V(H)|$ is a multiple of 3.  Since $|V(H)|$ is even and $H$ is distinct from $\overline{C_6 }$, we have $|V(H)|\geq12$. The assertion holds.

Since $G$ is not minimal matching covered,
by (i), we have $b^{*}(G)\geq1$. So $G$ has a brick $G_1$ which is different from $K_4$ and $\overline{C_6 }$. Again by (i), we have $|RE(G)|\geq|V(G_1)|$. Recall that every brick of $G$ is cubic claw-free, so is $G_1$.
By the above claim, we have $|V(G_1)|\geq12$, so $|RE(G)|\geq|V(G_1)|\geq12$.

We  next assert  that the above bound  is sharp for an infinite number of cubic claw-free matching covered graphs. Let $H_1$ be the graph obtained from $K_4$ by replacing each vertex $v$ of $K_4$ with a triangle and joining the three neighbours of $v$ to the three vertices in the triangle, respectively. Then  $H_1$ is a 3-connected cubic claw-free graph. By Lemma \ref{3-connected-cubic-claw-free}, we have $|RE(H_1)|=|V(H_1)|=12$.
Let $H_2$ be a cubic graph such that $H_2\in\mathcal{F}$.
By Theorem \ref{claw-free-mc}, $H_2$ is a cubic claw-free minimal matching covered graph, so  $RE(H_2)=\emptyset$.  Let $e_i$ be a ridge of $H_i$, $i=1,2$.
By Lemma \ref{cubic-ridge-nonremovable}, $e_i$ is nonremovable in $H_i$. Note that $e_i$ is a bisimplicial edge of $H_i$.
Let $H=(H_1\oplus H_2)_{(e_1,e_2)}$. Then $H$ is cubic because both $H_1$ and $H_2$ are cubic. Moreover, by Lemma \ref{four-operation-property}, $H$ is a claw-free matching covered graph, and
$$|RE(H)|=|(RE(H_1)\cup RE(H_2))\backslash\{e_1,e_2\}|=|RE(H_1)|=12.$$
The assertion holds.
Note that $|V(H_2)|$ can be infinitely large. The above bound is the best possible.
This completes the proof of Theorem \ref{2-connected-cubic-claw-free}.  $\hfill\square$

\section*{Acknowledgements}
This work is supported by the National Natural Science Foundation of China under grant
numbers 12171440, 12371318 and 12371361.

\end{document}